\documentclass[12pt]{amsart}
\advance\hoffset by -0.5 truecm
\usepackage{amsmath,amscd}
\usepackage{amssymb}
\usepackage[bbgreekl]{mathbbol}
\usepackage{amsthm}
\usepackage{array}
\usepackage{graphicx}
\usepackage{color}
\DeclareSymbolFontAlphabet{\mathbb}{AMSb}
\DeclareSymbolFontAlphabet{\mathbbl}{bbold}

\usepackage{mathrsfs}
\usepackage{hyperref}

\usepackage[colorinlistoftodos]{todonotes}

\usepackage{comment}

\usepackage{calc}
             {\begin{list}{\arabic{enumi}.}{\usecounter{enumi}%
              \setlength{\labelsep}{0.5em}%
              \settowidth{\labelwidth}{\arabic{enumi}.}%
              \setlength{\leftmargin}{\labelwidth+\labelsep}}}%
             {\end{list}}

\newtheorem{Theorem}{Theorem}[section]

\newtheorem{Lemma}[Theorem]{Lemma}

\newtheorem{Proposition}[Theorem]{Proposition}

\theoremstyle{definition}

\newtheorem{Definition}[Theorem]{Definition}

\newtheorem*{ExampleNoNumber}{Example}

\newtheorem{Remark}[Theorem]{Remark}

\numberwithin{equation}{section}

\def \dim{{\mbox {dim}}\,}

\def\V{\mbox{Var}}

\def\R\re
\def\V{\bf V}

\def \re{{\mathbb R}}
\def \mR{{\mathbb R}}

\def \mC{{\mathbb C}}

\def \V{{\bf V}}

\newcommand{\mc}{\mathcal}
\newcommand{\pl}{\partial}

\newcommand{\abs}[1]{\lvert #1 \rvert}
\newcommand{\norm}[1]{\lVert #1 \rVert}

\newcommand{\etilde}{\,\tilde{\rule{0pt}{6pt}}\,}

\newcommand{\eps}{\varepsilon}

\newcommand{\p}{\partial}

\def \vd{\overset{\tt{v}}{\nabla}}
\def \hd{\overset{\tt{h}}{\nabla}}
\def \vdiv{\overset{\tt{v}}{\mbox{\rm div}}}
\def \hdiv{\overset{\tt{h}}{\mbox{\rm div}}}

\def \vdelta{\overset{\tt{v}}{\delta}}
\def \ehd{\overset{\smash{\tt{h}}}{\overline{\nabla}}}
\def \ehdp{\overset{\smash{\tt{h}}}{\overline{\nabla}} {\hspace{.5pt}}^{\parallel}}
\newcommand{\ip}[2]{\left\langle#1,#2\right\rangle}
\newcommand{\vnu}{\ip{v}{\nu}}

\newcounter{sidenote}
\begin{document}

\title[A Sharp stability estimate for tensor tomography]{A Sharp stability estimate for tensor tomography in non-positive curvature} 

\author[G.P. Paternain]{Gabriel P. Paternain}
\address{ Department of Pure Mathematics and Mathematical Statistics,
University of Cambridge,
Cambridge CB3 0WB, UK}
\email {g.p.paternain@dpmms.cam.ac.uk}

\author[M. Salo]{Mikko Salo}
\address{Department of Mathematics and Statistics, University of Jyv\"askyl\"a}
\email{mikko.j.salo@jyu.fi}

%\subjclass{53C25, 53C21, 58F17, 35J15}

%\date{October 27, 2014}

%\maketitle

\begin{abstract} We consider the geodesic X-ray transform acting on solenoidal tensor fields on a compact simply connected manifold with strictly convex boundary and non-positive curvature. We establish a stability estimate
of the form $L^2\mapsto H^{1/2}_{T}$, where the $H^{1/2}_{T}$-space is defined using the natural parametrization of geodesics as initial boundary points and incoming directions (fan-beam geometry); only tangential derivatives at the boundary are used. The proof is based on the Pestov identity with boundary term localized in frequency.
%In \cite{PS18} we observed that for closed manifolds the Pestov identity localizes in frequency and here we show that the boundary term also has this property. This allows us to get rid of vertical derivatives and obtain the sharp stability estimate thus answering a questions raised by Boman and Sharafutdinov in \cite{BS18} and Assylbekov and Stefanov in \cite{AS19}.

\end{abstract}

\maketitle

%\tableofcontents

\section{Introduction}

To motivate our results, let us begin with the simplest case of the Radon transform in $\mR^2$ in parallel beam geometry (see \cite{Natterer} for more details).

\begin{ExampleNoNumber}
If $f \in C^{\infty}_c(\mR^2)$, the Radon transform of $f$ is 
\[
Rf(s,v) = \int_{-\infty}^{\infty} f(s v + t v^{\perp}) \,dt, \qquad s \in \mR, \ v \in S^1,
\]
where $v^{\perp}$ is the rotation of $v$ by $90^{\circ}$ counterclockwise. The Fourier transform of $Rf$ in the $s$ variable, denoted by $(Rf)\etilde(\,\cdot\,,v)$, satisfies the \emph{Fourier slice theorem} 
\[
(Rf)\etilde(\sigma, v) = (2\pi)^{1/2} \hat{f}(\sigma v), \qquad \sigma \in \mR, \ v \in S^1.
\]
Using the Plancherel theorem and polar coordinates, we obtain that 
\begin{align*}
\norm{f}_{L^2(\mR^2)}^2 &= \norm{\hat{f}}_{L^2(\mR^2)}^2 = \int_0^{\infty} \int_{S^1} \abs{\hat{f}(\sigma v)}^2 \sigma \,dv \,d\sigma \\
 &= \frac{1}{2} \int_{-\infty}^{\infty} \int_{S^1} \abs{\hat{f}(\sigma v)}^2 \abs{\sigma} \,dv \,d\sigma \\
 &= \frac{1}{4 \pi} \int_{-\infty}^{\infty} \int_{S^{1}}\abs{(Rf)\etilde(\sigma, v)}^2 \abs{\sigma} \,dv \,d\sigma.
\end{align*}
In particular, this implies the stability estimate 
\begin{equation}
\norm{f}_{L^2(\mR^2)} \leq \frac{1}{(4\pi)^{1/2}} \norm{Rf}_{H^{1/2}_T(\mR \times S^1)}
\label{eq:radon}
\end{equation}
with the mixed Sobolev norm $\norm{h}_{H^{1/2}_T(\mR \times S^1)} = \norm{(1+\sigma^2)^{1/4} \tilde{h}(\sigma,v)}_{L^2(\mR \times S^1)}$.
\end{ExampleNoNumber}

The main question we address in the present paper is the existence of a stability estimate analogous to \eqref{eq:radon} but in a geometric setting, namely, when $\mR^2$ and the lines in the plane are replaced
by a Riemannian manifold and its geodesics. There are two features we wish to preserve from \eqref{eq:radon}: one is its $L^2\to H^{1/2}$ nature and the other is that the $H^{1/2}_{T}$ only incorporates ``half of the derivatives" of the target space (space of geodesics).

Let us first be more precise about the geometric setting. The geodesic X-ray transform acts on functions defined on the unit sphere bundle of a compact oriented $d$-dimensional Riemannian manifold $(M,g)$ with smooth boundary $\pl M$ ($d\geq 2$).
Let $SM$ denote the unit sphere bundle on $M$, i.e.
$$SM:=\{(x,v)\in TM : |v|_g=1\}.$$
We define the volume form on $SM$ by $d\Sigma^{2d-1}(x,v)=dV^d(x)\wedge dS_x(v)$, where $dV^d$ is the volume form on $M$ and $dS_x$ is the volume form on the fibre $S_xM$.
The boundary of $SM$ is $\pl SM:=\{(x,v)\in SM : x\in \pl M\}$. On $\pl SM$ the natural volume form is $d\Sigma^{2d-2}(x,v)=dV^{d-1}(x)\wedge dS_x(v)$, where $dV^{d-1}$ is the volume form on $\pl M$. We distinguish two subsets of $\pl SM$ (incoming and outgoing directions)
$$\pl_{\pm}SM:=\{(x,v)\in \pl SM : \pm\langle v,\nu(x)\rangle_g\leq 0\},$$
where $\nu(x)$ is the outward unit normal vector on $\pl M$ at $x$. It is easy to see that
$$\pl_+SM\cap\pl_-SM=S(\pl M).$$
Given $(x,v)\in SM$, we denote by $\gamma_{x,v}$ the unique geodesic with $\gamma_{x,v}(0)=x$ and $ \dot{\gamma}_{x,v}(0)=v$ and let $\tau(x,v)$ be the first time when the geodesic $\gamma_{x,v}$  exits $M$. 

We say that $(M,g)$ is {\it non-trapping} if $\tau(x,v)<\infty$ for all $(x,v)\in SM$. In this case the space of geodesics is naturally parametrized by $\partial_{+}SM$ (fan-beam geometry).

\begin{Definition}
The \emph{geodesic X-ray transform} of a function $F \in C^{\infty}(SM)$ is the function 
\begin{equation*}
IF(x,v):=\int\limits_{0}^{\tau(x,v)}F(\gamma_{x,v}(t),\dot{\gamma}_{x,v}(t)) \,dt,\quad
(x,v)\in \partial_{+} SM.
\end{equation*}
\end{Definition}

If the manifold $(M,g)$ is non-trapping and has strictly convex boundary, then $\tau\in C^{\infty}(\partial_{+}SM)$ and as a consequence $I:C^{\infty}(SM)\rightarrow C^{\infty}(\partial_{+} SM)$. Moreover, $I$ extends as a bounded operator $I:H^{k}(SM)\to H^{k}(\partial_{+}SM)$ for all $k\geq 0$ \cite[Theorem 4.2.1]{Sh}, where the Sobolev spaces are defined using the $L^2$-inner  products arising from the volume forms introduced above.

We shall consider $I$ acting on special functions $F\in C^{\infty}(SM)$ induced by symmetric tensor fields.
We denote by $C^{\infty}(S^m(T^*M))$ the space of smooth covariant symmetric tensor fields of rank $m$ on $M$
with $L^2$ inner product:
$$(u,w):=\int_M u_{i_1\cdots i_m}w^{i_1\cdots i_m} \,dV^d,$$
where $w^{i_1\cdots i_m}=g^{i_1j_1}\cdots g^{i_mj_m} w_{j_1\cdots j_m}$. There is a natural 
map 
$$\ell_{m}: C^{\infty}(S^{m}(T^*M))\to C^{\infty}(SM)$$ given by $\ell_{m}(f)(x,v):=f_{x}(v,\dots,v)$.
We can now define the geodesic ray transform acting on symmetric $m$-tensors simply by
setting $I_{m}:=I\circ \ell_{m}$.
Let $d^s=\sigma\nabla$ be the symmetric inner differentiation, where $\nabla$ is the Levi-Civita connection associated with $g$, and $\sigma$ denotes symmetrization. 
It is easy to check that if $f=d^s p$ for some $p\in C^{\infty}(S^{m-1}(T^*M))$ with $p|_{\pl M}=0$, then $I_{m}f=0$. The tensor tomography problem asks the following question: are such tensors the only obstructions for $I_{m}$ to be injective? If this is the case, then we say $I_m$ is {\it solenoidal injective or s-injective} for short. 
This terminology is explained by the following well known decomposition (cf.\ \cite{Sh}).
Given $f\in H^k(S^m(T^*M)),\, k\geq 0$, there exist uniquely determined $f_s\in H^k(S^m(T^*M))$ and $p\in H^{k+1}(S^{m-1}(T^*M))$, such that 
$$f=f_s+d^sp,\quad \delta^s f_s=0,\quad p|_{\pl M}=0,$$
where $\delta^s$ is the divergence. We call $f_s$ and $d^s p$ the {\it solenoidal} part and {\it potential} part of $f$ respectively.

There is one important instance in which the tensor tomography problem is solved for tensors of any order $m$ and in any dimension $d$. This is when we assume in addition that the sectional curvature of $M$ is non-positive.
Moreover, in this case a stability estimate is available as follows:

\begin{Theorem}{{\rm (}\cite{PS} and \cite[Theorem 4.3.3]{Sh}{\rm )}} \label{thm:Sh}
Let $(M,g)$ be a simply connected compact manifold with strictly convex boundary and
non-positive sectional curvature. Given $m\geq 0$ there is a constant $C>0$ such that for any $f\in H^{1}(S^{m}(T^*M))$
\[\norm{f_{s}}^2_{L^{2}}\leq C(\norm{I_m f}^2_{H^{1}(\partial_{+}SM)}+m\norm{f}_{H^{1}}\norm{I_m f}_{L^{2}}).\]
\end{Theorem}

(We note that a manifold as in the theorem is necessarily non-trapping.)
There are two notorious differences between the stability estimate above and that in \eqref{eq:radon}.
Firstly, the stability estimate in Theorem \ref{thm:Sh} has in the right hand side the term $\norm{f}_{H^{1}}\norm{I_{m}f}_{L^{2}}$ when $m\neq 0$. Secondly, it is not sharp in the sense that it is $L^2\to H^1$.
In \cite{BS18} Boman and Sharafutdinov resolved these issues for strictly convex domains in Euclidean space and asked whether the same was true for the more general setting of non-positively curved Riemannian manifolds. This paper provides a positive answer to these questions.
Moreover, the $1/2$-Sobolev space on the target space of $I_m$ is naturally suggested by the geometry
and the most relevant $L^2$-energy identity for the problem: {\it the Pestov identity}.
The Pestov identity with boundary term in the way that we shall use it here was derived for instance in
\cite[Lemma 8]{IP18}. It contains a boundary term given by
\begin{equation}
(Tu,\vd u)_{L^{2}(\partial SM)}
\label{eq:bt}
\end{equation}
where $u\in C^{\infty}(\partial SM)$, $\vd$ is the vertical gradient, and $T$ is a {\it tangential operator} defined by 
\[Tu:=\langle \nu(x),v\rangle\ehd u-\nu Xu,\]
where $X$ is the geodesic vector field and $\ehd$ the full horizontal gradient (we refer to Sections \ref{section:prelim} and \ref{section_pestov_general_connection} for the precise definitions). The operator $T$ acts on $\partial SM$ and it only involves horizontal derivatives. This suggests that only horizontal derivatives of $I_m f$ on $\partial SM$ should appear in the stability estimate.

We can define the {\it tangential} (or \emph{horizontal}) $H^{1}(\partial SM)$-norm by setting
\[\norm{u}^2_{H^{1}_{T}(\p SM)}:=\norm{u}^{2}_{L^{2}(\p SM)}+\norm{\ehdp u}^{2}_{L^{2}(\p SM)}\]
where $\ehdp u$ contains the tangential derivatives in $\ehd u$ along $\p M$. For example, if $M$ is a ball in $\mR^p$ with Euclidean metric, then $\p SM = \p M \times S^{d-1}$ and 
\[
\norm{u}_{H^1_T(\p SM)}^2 = \int_{\p M} \int_{S^{d-1}} (\abs{u(x,v)}^2 + \abs{\nabla_x u(x,v)}^2) \,dS(v) \,dV(x)
\]
where $\nabla_x$ is the gradient on $\p M$. The space $H^{1}_{T}(\p_+ SM)$ is defined by restriction, and $H^{1/2}_{T}(\partial _{+}SM)$ is defined by complex interpolation between $L^2(\partial_{+}SM)$ and $H^{1}_{T}(\partial_{+}SM)$.

With this definition we may now state our main result:

\begin{Theorem} \label{thm:main_final}
Let $(M,g)$ be a simply connected compact manifold with strictly convex boundary and
non-positive sectional curvature. Given $m\geq 0$ there is a constant $C>0$ such that for any $f\in H^{1}(S^{m}(T^*M))$
\[\norm{f_{s}}_{L^{2}}\leq C\norm{I_m f}_{H_{T}^{1/2}(\partial_{+}SM)}.\]
\end{Theorem}

The constant $C$ can be estimated in terms of $m$ and $(M,g)$. In fact, for the related stability result for the transport equation in Theorem \ref{thm_stability_transport}, one can take $C = 1$.

Most of work in the proof of Theorem \ref{thm:main_final} lies in the upgrade from the $H^{1}(\partial_{+}SM)$-norm in
Theorem \ref{thm:Sh} to the $H^{1/2}_{T}(\partial_{+}SM)$-norm. The upgrade is possible thanks to the localization in frequency of the Pestov identity first noted in full generality in \cite{PS18} (in two dimensions this was proved in \cite{PSU_hd}). However, in \cite{PS18} we did not consider the boundary term.
It turns out, quite remarkably, that the boundary term \eqref{eq:bt} also localizes in frequency. This allows us to change the norm for $I_m f$ from $H^1$ to $H^{1/2}_T$, thus producing the upgrade. We also mention that for $d=2$ the proof would simplify substantially because spherical harmonics decompositions and the $T$ operator are simpler; the two-dimensional proof will be given in \cite{PSU_book}.

\subsection{Related results and alternative approaches} There are many earlier results on stability estimates for $I_m$, using different techniques. One approach is to consider the normal operator $I^*_{m}I_{m}$ where the adjoint $I^*_{m}$ is computed using a natural $L^2_{\mu}$-inner product on $\partial_{+}SM$ suggested by the Santal\'o formula. When $M$ is free of conjugate points, it turns out that $I^*_{m}I_{m}$ is an (elliptic) $\Psi$DO of order $-1$ on a slightly larger open manifold engulfing $M$. This approach has produced stability estimates for the normal operator, cf.\ \cite{SU04}, and has proved to be of fundamental importance in the solution of several geometric inverse problems. One drawback is that one needs to work on the slightly extended manifold, unless one is willing to incorporate modified transmission conditions to account for boundary effects \cite{MNP}.  Another drawback is that the approach does not give estimates for the constants due to a compactness argument. Still, quite recently, a sharp stability estimate has been obtained in \cite{AS19}, by defining a suitable $H^{1/2}$-norm based on this extension or equivalently on a different parametrization of the space of geodesics.  Our approach in Theorem \ref{thm:main_final} deals directly with the boundary and with the space of geodesics in ``fan-beam" geometry as given by $\partial_{+}SM$. In this sense our theorem addresses the open problem stated at the end of the introduction in \cite{AS19}. Also our tangential derivatives are naturally suggested by the geometry of the problem.

The microlocal approach can actually be pushed further, using scattering calculus and a combination of a local theorem with a global strict convexity assumption as in \cite{UhlmannVasy,SUV_tensor}. This is also very powerful, and allows even to consider situations with conjugate points as long as $d\geq 3$. However, the stability estimates thus produced are $L^2\to H^1$.

One drawback of Theorem \ref{thm:main_final} is the curvature assumption. In \cite{AS19} the estimates hold for compact simple manifolds for $m=0,1$ and for $m=2$ when $I_{m}$ is known to be injective, e.g.\ when $d=2$ \cite{Sh07,PSU1}. Another possible improvement would be to replace the assumption $f \in H^1$ by $f \in L^2$ and to prove the two-sided inequality 
\[c \norm{f_{s}}_{L^{2}}\leq \norm{I_m f}_{H_{T}^{1/2}(\partial_{+}SM)} \leq C \norm{f_{s}}_{L^{2}}.
\]
For this, one would like to prove that $I_m$ is bounded from $L^2$ to $H^{1/2}_T$. This is true if $f$ vanishes near $\p M$ since $I_m$ is a Fourier integral operator, but it is not clear how to prove this with uniform bounds when the support of $f$ extends up to $\p M$.

Finally, we mention that quite recently, Monard \cite{M19} has studied very detailed mapping properties of $I_0$ for 2D discs of constant curvature at all Sobolev scales; for these cases, he also obtains a stability estimate with a suitable $H^{1/2}$-norm. Further references to stability estimates for $I_m$ may be found in \cite{AS19}.

It is natural to speculate whether Theorem \ref{thm:main_final} extends to more general situations. For example
one could try to relax the assumption about strict convexity of the boundary as in \cite{GMT} or allow for
some hyperbolic trapping as in \cite{Gu15}. One could also speculate on extensions that include attenuations and sections of a suitable vector bundle. We do not pursue these here.

\subsection*{Acknowledgements} We are very grateful to the referee for several comments that improved the presentation and in particular for suggesting a simplified proof of Lemma \ref{lemma_localization_newproof}. GPP was supported by EPSRC grant EP/R001898/1 and the Leverhulme trust.
MS\ was supported by the Academy of Finland (Finnish Centre of Excellence in Inverse Modelling and Imaging, grant numbers 312121 and 309963) and by the European Research Council under Horizon 2020 (ERC CoG 770924). This material is based upon work supported by the National Science Foundation under Grant No.\ 1440140, while the authors were in residence at MSRI in Berkeley, California, during the semester on Microlocal Analysis in 2019.

%The problem is wide open for
%compact non-trapping manifolds with strictly convex boundary (but see \cite{UV,SUV}). There are more results if one assumes the stronger condition of being {\it simple}, i.e., $(M,g)$ is simply connected, has no conjugate points and strictly convex boundary.  For simple surfaces, the tensor tomography problem has been completely solved \cite{PSU1}. For simple manifolds of any dimension, solenoidal injectivity is known for $I_{0}$ and $I_{1}$ \cite{Mu, AR}.
% For $m$-tensors, $m\geq 2$, the tensor tomography problem is still open, but some substantial partial results were established under additional assumptions, see e.g. \cite{PS, Sh1, SU1, PSU5, SUV}. 

\section{Geometric preliminaries} \label{section:prelim}

In this section we collect some geometric preliminaries for subsequent use. \\[-5pt]

\noindent {\bf Unit sphere bundle.}
We start by recalling some standard notions related to the geometry of the unit sphere bundle. We follow the setup and notation of \cite{PSU_hd}; for other approaches and background information see \cite{GK2,Sh,Pa,Kn,DS}.

Let $(M,g)$ be a $d$-dimensional compact Riemannian manifold with or without boundary, having unit sphere bundle $\pi: SM\to M$, and let $X$ be the geodesic vector field. We equip $SM$ with the Sasaki metric. If $\mathcal V$ denotes the vertical subbundle
given by $\mathcal V=\mbox{\rm Ker}\,d\pi$, then there is an orthogonal splitting with respect to the Sasaki metric:
\begin{equation}\label{TSM}
TSM=\re X\oplus {\mathcal H}\oplus {\mathcal V}.
\end{equation}
The subbundle ${\mathcal H}$ is called the horizontal subbundle. Elements in $\mathcal H(x,v)$ and $\mathcal V(x,v)$ are canonically identified with elements in the codimension one subspace $\{v\}^{\perp}\subset T_{x}M$ by the isomorphisms
\[ d\pi_{x,v} : \mc{V}(x,v)\to \{v\}^{\perp} ,  \quad \mc{K}_{x,v}: \mathcal H(x,v)\to \{v\}^{\perp},\]
here $\mc{K}_{x,v}$ is the connection map coming from Levi-Civita connection.
We will use these identifications freely below.  

We shall denote by $\mathcal Z$ the set of smooth functions $Z:SM\to TM$ such that $Z(x,v)\in T_{x}M$ and $\langle Z(x,v),v\rangle=0$ for all $(x,v)\in SM$.
Alternatively we may describe the elements of $\mathcal Z$ is a follows. Consider the pull-back bundle $\pi^*TM$ over $SM$ and let $N$ denote the subbundle of $\pi^*TM$ whose fiber over $(x,v)$
is given by $N_{(x,v)}=\{v\}^{\perp}$. Then $\mathcal Z$ coincides with the smooth sections
of the bundle $N$. Note that $N$ carries a natural scalar product and thus an $L^{2}$-inner product 
(using the Liouville measure on $SM$ for integration).

Given a smooth function $u\in C^{\infty}(SM)$ we can consider its gradient $\nabla u$ with respect to the Sasaki metric. 
Using the splitting above we may write uniquely in the decomposition \eqref{TSM}
\[\nabla u=((Xu)X,\hd u,  \vd u). \]
The derivatives $\hd u\in  \mc{Z}$ and $\vd u\in \mc{Z}$ are called horizontal and vertical derivatives respectively. (This differs from the definitions in \cite{Kn,Sh} since here all objects are defined on $SM$ as opposed to $TM$.)

The geodesic vector $X$ acts on $\mathcal Z$ as follows:
\begin{equation}\label{XonZ}
XZ(x,v):=\frac{DZ(\varphi_{t}(x,v))}{dt}|_{t=0}
\end{equation}
where $D/dt$ is the covariant derivative with respect to Levi-Civita connection and $\varphi_t$ is the geodesic flow.  With respect to the $L^2$-product on $N$, the formal adjoints of $\vd:C^{\infty}(SM)\to\mathcal Z$ and $\hd:C^{\infty}(SM) \to \mathcal Z$ are denoted by $-\vdiv$ and $-\hdiv$ respectively. Note that since $X$ leaves invariant the volume form of the Sasaki metric we have $X^*=-X$ for both actions of $X$ on $C^{\infty}(SM)$ and $\mathcal Z$.
%In what follows, we will need to work with the complexified version of $N$ with its natural inherited Hermitian product. This will be clear from the context and we shall employ the same letter $N$ to denote the complexified
%bundle and also $\mathcal Z$ for its sections.

Let $R(x,v):\{v\}^{\perp}\to \{v\}^{\perp}$ be the operator determined by the Riemann curvature tensor by $R(x,v)w=R(w,v)v$, and let $d=\dim M$. \\[-5pt]

\noindent {\bf Spherical harmonics decomposition.}
There is a natural spherical harmonics decomposition with respect to the vertical Laplacian $\Delta = -\vdiv \vd$ (cf.\ \cite[Section 3]{PSU_hd} and \cite{GK2}):
$$
L^2(SM) = \bigoplus_{m=0}^{\infty} H_m(SM),
$$
so that any $f \in L^2(SM)$ has the orthogonal decomposition 
$$
f = \sum_{m=0}^{\infty} f_m.
$$
We write $\Omega_m = H_m(SM) \cap C^{\infty}(SM)$. Then $\Delta u = m(m+d-2)u$ for $u \in \Omega_m$ and we let $\lambda_{m}:=m(m+d-2)$.
\vspace{10pt}

\noindent {\bf Decomposition of $X$.} The geodesic vector field has a special behaviour with respect to the decomposition into fibrewise
spherical harmonics: it maps $\Omega_{m}$ into $\Omega_{m-1}\oplus\Omega_{m+1}$ \cite[Proposition 3.2]{GK2}. Hence on $\Omega_{m}$ we can write
\[X=X_{-}+X_{+}\] 
where $X_{-}:\Omega_{m}\to \Omega_{m-1}$ and $X_{+}:\Omega_{m}\to\Omega_{m+1}$.
By \cite[Proposition 3.7]{GK2} the operator $X_{+}$ is overdetermined elliptic (i.e.\ it has injective principal symbol). We can explain the decomposition $X=X_{-}+X_{+}$ as follows. Fix $x\in M$ and consider local coordinates which are geodesic at $x$ (so\ $\p_{x_j} g_{kl}(x) = 0$ for all $j,k,l$).
Then $Xu(x,v)=v^{i}\frac{\partial u}{\partial x^{i}}$. We now use the following basic fact about spherical harmonics: the product of a spherical harmonic of degree $m$ with a spherical harmonic of degree one decomposes as the sum of spherical harmonics of degree $m-1$ and $m+1$.

\section{Pestov identity with boundary term} \label{section_pestov_general_connection}

We recall  the following commutator formulas from \cite{PSU_hd}:
%The following commutator formulas hold on $C^{\infty}(SM)$:
\begin{equation} \label{basic_commutator_formulas}
\begin{split}
[X,\vd]&=-\hd, \\ %\label{eq:commXV}\\
[X,\hd]&=R\vd, \\ %\label{eq:commXH}\\
\hdiv\vd-\vdiv\hd&=(d-1)X. %\label{eq:commdiv}
\end{split}
\end{equation}
Taking adjoints gives the following commutator formulas on~$\mathcal Z$:
\begin{equation}
\begin{split}
[X,\vdiv] &= -\hdiv, \\
[X,\hdiv] &= -\vdiv R.
\end{split}
\end{equation}

Using these relations one can establish a Pestov identity with boundary term.
Let $\mu(x,v):=\langle v,\nu(x)\rangle$.  We let $\norm{\cdot}$ and $(\cdot,\cdot)$ denote the $L^2$-norm and $L^{2}$-inner product respectively determined
by the volume form $d\Sigma^{2d-1}$ on $SM$; we let $(\cdot,\cdot)_{\partial SM}$ stand for the $L^{2}$-inner product
on $\partial SM$ determined by $d\Sigma^{2d-2}$.

\begin{Proposition}[Pestov identity with boundary term, cf.\ Lemma 8 in \cite{IP18}] \label{prop_pestov_boundary}
Let $(M,g)$ be a compact manifold with smooth boundary. If $u \in C^{\infty}(SM)$, then
\[
\norm{\vd Xu}^{2}=\norm{X\vd u}^2 - (R \vd u, \vd u) + (d-1)\norm{Xu}^{2}+P(u,u),
\]
where $P$ is the quadratic form defined by
\[P(u,w)=(Tu,\vd w)_{\partial SM},\]
and $Tu:=\mu\hd u-Xu\vd\mu$.
\end{Proposition}

We can express $T$ using the full horizontal derivative $\ehd u=\hd u+(Xu)v$ as
$Tu=\mu\ehd u-\nu Xu$ since $\vd\mu=\nu-\mu v$.  It turns out that $T$ can also be rewritten
in such a way that it acts on functions $u\in C^{\infty}(\partial SM)$. To see this,
consider the operators
\begin{equation}
\label{eq:d-par-def}
\nabla^{\parallel}u:=
\nabla u-\ip{\nabla u}{(\nu,0)}(\nu,0)
\end{equation}
and
\begin{equation}
\label{eq:X-par-def}
X^{\parallel}:=
X-\ip{X}{(\nu,0)}(\nu,0)
=
(v-\vnu\nu,0)
=
(v^{\parallel},0)
\end{equation}
at the boundary. Note that $(\nu,0)$ is the horizontal lift of $\nu$. We also define the horizontal part of $\nabla^{\parallel}$ as 
\[
\ehdp u := d\pi(\nabla^{\parallel} u) = \ehd u - \langle \ehd u, \nu \rangle \nu.
\]
The following simple lemma is proved in \cite[Lemma 14]{IP18}:

\begin{Lemma}
\label{lma:Q}
We have
\begin{equation}
%\vnu\ehd-\nu X
Tu
=\mu \ehdp u -\nu X^{\parallel}u.
\label{eq:Tfull}
\end{equation}
\end{Lemma}

%\begin{proof}
%The operator~$\ehd$ is just~$\der\pi\nabla$, so applying~$\der\pi$ to the projected operator of~\eqref{eq:d-par-def} we derive
%\begin{equation}
%\der\pi\nabla^{\parallel}u
%=
%\ehd u-\ip{\ehd u}{\nu}\nu.
%\end{equation}
%Similarly, we also project $X=(v,0)$ using~\eqref{eq:X-par-def} and find
%\begin{equation}
%Xu
%=
%X^{\parallel}u-\vnu \ip{\ehd u}{\nu}
%\end{equation}
%and the claim follows.
%\end{proof}
From this form we can clearly see that~$T: C^{\infty}(\partial SM)\to {\mathcal Z}|_{\partial SM}$.

\begin{Remark} In 2D, $Tu=(\mathbb{T}u)iv$, where $\mathbb T$ is the tangential horizontal vector field $(i\nu,0)$ and $i$ is the complex structure of the surface. The vector field $\mathbb{T}$ and the vertical vector field $V$ form a commuting frame for $\partial SM$.
\end{Remark}

We next rewrite the Pestov identity in terms of $X_+$ and $X_-$ as in \cite{PS18}. To do this, we need some notation: for a polynomially bounded sequence $\alpha = (\alpha_l)_{l=0}^{\infty}$ of real numbers, we define a corresponding ''inner product'' 
\[
(u,w)_{\alpha} = \sum_{l=0}^{\infty} \alpha_l (u_l, w_l)_{L^2(SM)}, \qquad u, w \in C^{\infty}(SM).
\]
We also write $\norm{u}_\alpha^2 = \sum_{l=0}^{\infty} \alpha_l \norm{u_l}^2$. (If each $\alpha_l$ is positive one gets an actual inner product and norm, but it is notationally convenient to allow zero or negative $\alpha_l$.)

The Pestov identity can then be written in the following form. Define
\begin{align}
\alpha_l &= \lambda_l \left[ \left(1+\frac{1}{l+d-2} \right)^2 - 1 \right] + (d-1), \label{alphal_formula} \\
\beta_l &= \lambda_l \left[ 1 -  \left(1-\frac{1}{l} \right)^2 \right]  - (d-1). \label{betal_formula}
\end{align}
The next result extends \cite[Proposition 4.4]{PS18} to the case with boundary terms.

\begin{Proposition}[Pestov identity in terms of $X_{\pm}$ with boundary term] \label{prop_pestov_xplus_xminus_boundary}
Let $(M,g)$ be a compact manifold with smooth boundary. If $u \in C^{\infty}(SM)$, then
\[
\norm{X_- u}_{\alpha}^2 - (R \vd u, \vd u) + \norm{Z(u)}^2+P(u,u) = \norm{X_+ u}_{\beta}^2,
\]
where $Z(u)$ is $\vdiv$-free.
\end{Proposition}

\begin{proof} Recall from \cite[Lemma 4.4]{PSU_hd} that 
\begin{equation} \label{hdau_decomposition}
\hd u = \vd \left[ \sum_{l=1}^{\infty} \left( \frac{1}{l} X_+ u_{l-1} - \frac{1}{l+d-2} X_- u_{l+1} \right) \right] + Z(u)
\end{equation}
where $Z(u) \in \mathcal Z$ satisfies $\vdiv \,Z(u) = 0$. Thus by \eqref{basic_commutator_formulas} 
\begin{equation}
X \vd u = \vd \sum_{l=1}^{\infty} \left[ \left(1-\frac{1}{l} \right) X_+ u_{l-1} + \left(1+\frac{1}{l+d-2} \right) X_- u_{l+1} \right] - Z(u).
\label{eq:XV}
\end{equation}
This gives 
{\scriptsize 
\begin{align*}
 &\norm{X\vd u}^{2}  \\
 & = \sum_{l=1}^{\infty} \lambda_l \left( \left(1-\frac{1}{l} \right) X_+ u_{l-1} + \left(1+\frac{1}{l+d-2} \right) X_- u_{l+1},  \left(1-\frac{1}{l} \right) X_+ u_{l-1} + \left(1+\frac{1}{l+d-2} \right) X_- u_{l+1} \right) \\
 &\qquad + \norm{Z(u)}^{2} \\
 &= \sum_{l=1}^{\infty} \lambda_l \left[�\left(1-\frac{1}{l} \right)^2 \norm{X_{+}u_{l-1}}^{2} + \left(1+\frac{1}{l+d-2} \right)^2 \norm{X_- u_{l+1}}^2 \right] \\
 & \qquad + \sum_{l=1}^{\infty} \lambda_l \left(1-\frac{1}{l} \right) \left(1+\frac{1}{l+d-2} \right) \left[ (X_+ u_{l-1}, X_- u_{l+1}) + (X_- u_{l+1}, X_+ u_{l-1}) \right] + \norm{Z(u)}^{2}.
\end{align*}
}
On the other hand, one has 
\begin{align*}
 &\norm{\vd X u}^2 - (d-1) \norm{X u}^{2} \\
 &= -(d-1) \norm{X_- u_1}^{2} + \sum_{l=1}^{\infty} (\lambda_l - (d-1)) (X_+ u_{l-1} + X_- u_{l+1}, X_+ u_{l-1} + X_- u_{l+1}) \\
 &= -(d-1) \norm{X_- u_1}^{2} + \sum_{l=1}^{\infty} (\lambda_l - (d-1)) \left[ \norm{X_+ u_{l-1}}^{2} + \norm{X_- u_{l+1}}^{2} \right] \\
 &\qquad + \sum_{l=1}^{\infty} (\lambda_l - (d-1)) \left[ (X_+ u_{l-1}, X_- u_{l+1}) + (X_- u_{l+1}, X_+ u_{l-1}) \right].
\end{align*}
Somewhat miraculously, we observe that 
\[
\lambda_l \left(1-\frac{1}{l} \right) \left(1+\frac{1}{l+d-2} \right) = \lambda_l - (d-1).
\]
This means that the two sums above involving $\left[ (X_+ u_{l-1}, X_- u_{l+1}) + (X_- u_{l+1}, X_+ u_{l-1}) \right]$ terms are equal. The Pestov identity from Proposition \ref{prop_pestov_boundary} now yields 

\[\sum_{l=0}^{\infty} \alpha_l \norm{X_- u_{l+1}}^{2} - (R \vd u, \vd u) +\norm{Z(u)}^{2}+P(u,u)= \sum_{l=1}^{\infty} \beta_l \norm{X_+ u_{l-1}}^{2}\]
where $\alpha_l$, $\beta_l$ are as in \eqref{alphal_formula}--\eqref{betal_formula}. The result follows.
\end{proof}

Later on we shall need the following useful property.

\begin{Lemma}[Adjoint of $T$] The formal adjoint of $T:C^{\infty}(\partial SM)\to {\mathcal Z}|_{\partial SM}$
satisfies
\[\vdiv T=-T^*\vd\]
and the operator $\vdiv T$ is self-adjoint in $L^{2}(\partial SM,d\Sigma^{2d-2})$.

\label{lemma:Qstar}
\end{Lemma}

\begin{proof} We use the Pestov identity with boundary term, to claim first
that the operator $\vdiv T$ is self adjoint. Proposition \ref{prop_pestov_xplus_xminus_boundary} and the polarization identity imply that 
\[
P(u,w) = (X_+ u, X_+ w)_{\beta} - (X_- u, X_- w)_{\alpha} + (R \vd u, \vd w) -(Z(u),Z(w))
\]
and since $R$ is symmetric, it follows that $P(u,w)=P(w,u)$. But $P(u,w)=-(\vdiv T u,w)$ and thus
$\vdiv T$ is self-adjoint. Hence
\[\vdiv T=(\vdiv T)^*=-T^*\vd\]
as desired.
\end{proof}

%\begin{Remark} It is possible to compute $T^*$, but it turns out we only need the claim above later on.
%Using the bracket relations one finds
%\[\vdiv Tu=\mu\hdiv\vd u-\langle X\vd u,\vd\mu\rangle\]
%and hence
%\[T^*Z=-\mu\hdiv+\langle XZ,\vd\mu\rangle.\]

%\end{Remark}

\section{Frequency localization}

Recall from Section \ref{section:prelim} that any $u \in C^{\infty}(SM)$ admits an $L^2$-orthogonal decomposition 
\[
u = \sum_{l=0}^{\infty} u_l, \qquad u_l \in \Omega_l,
\]
where $\Omega_l$ corresponds to the set of vertical spherical harmonics of degree $l$. Since $X_{\pm}$ maps $\Omega_l$ to $\Omega_{l \pm 1}$, it is immediate that the Pestov identity with boundary term (Proposition \ref{prop_pestov_xplus_xminus_boundary}) reduces to the following identity when applied to functions in $\Omega_l$ (i.e.\ frequency localized Pestov identity).

\begin{Proposition}[Pestov identity on $\Omega_l$ with boundary term] \label{prop_pestov_localized}
Let $(M,g)$ be a compact manifold with smooth boundary, and let $l \geq 0$. One has 
\[
\alpha_{l-1} \norm{X_- u}^2 - (R \vd u, \vd u) + \norm{Z(u)}^2+P(u,u) = \beta_{l+1} \norm{X_+ u}^2, \qquad u \in \Omega_l.
\]
(We define $\alpha_{-1} = 0$.)
\end{Proposition}

It was proved in \cite{PS18} (and in \cite[Appendix B]{PSU_hd} when $\dim(M) = 2$) that the frequency localized Pestov identity for all $l$ is equivalent with the standard Pestov identity. The same is true in the boundary case:

\begin{Lemma} \label{lemma_pestov_equivalence}
The Pestov identity with boundary term on $\Omega_l$ is equivalent with the Pestov identity with boundary term in the following sense: for any $u \in C^{\infty}(SM)$, one has 
\begin{multline*}
\sum_{l=0}^{\infty} \left[ \alpha_{l-1} \norm{X_- u_l}^2 - (R \vd u_l, \vd u_l)+ \norm{Z(u_l)}^2+P(u_{l},u_{l}) -\beta_{l+1} \norm{X_+ u_l}^2 \right] \\
 = \norm{X_- u}_{\alpha}^2 - (R \vd u, \vd u) + \norm{Z(u)}^2+P(u,u) - \norm{X_+ u}_{\beta}^2.
\end{multline*}
\end{Lemma}

The result will follow if we can show that the curvature, $Z$ and $P$ terms localise. Thus Lemma \ref{lemma_pestov_equivalence} is a corollary of the next result.

\begin{Lemma} \label{lemma_curvature_fourier_localization}
If $(M,g)$ is a compact Riemannian manifold, then 
\[
(R\vd u, \vd w) = 0, \qquad (Z(u), Z(w)) = 0, \qquad P(u,w)=0
\]
whenever $u \in \Omega_m$, $w \in \Omega_l$ and $m \neq l$. In particular 
\[
\vdiv T: \Omega_m \to \Omega_m.
\]
\end{Lemma}

\begin{proof}The localization of the curvature term was proved \cite[Lemma 5.4]{PS18}. We shall prove here that the $Z$-term localizes. That is enough to obtain also the conclusion for $P$ since Proposition \ref{prop_pestov_xplus_xminus_boundary} and the polarization identity imply that 
\[
P(u,w) = (X_+ u, X_+ w)_{\beta} - (X_- u, X_- w)_{\alpha} + (R \vd u, \vd w) -(Z(u),Z(w)).
\]
Hence the statements for the curvature and $Z$-term imply that $P(u,w) = 0$ when $m \neq l$. The last claim follows since $P(u,w) = -(\vdiv T u, w)$.

The claim for $Z(u)$ for $d=2$ follows from \cite[Remark 6.5]{PS18} using the explicit representation for $Z(u)$. To prove the claim when $d \geq 3$, recall that $Z(u)$ is the  $\vdiv$-free part of $\hd u$ (the $\vdiv$-free part is uniquely defined since there are no nontrivial harmonic $1$-forms on $S_x M$ when $d \geq 3$). Using  the bracket relation
$\hd=\vd X-X\vd$ we can relate $X\vd$ and $Z(u)$. Indeed this is done explicitly in 
equation \eqref{eq:XV}, which shows that $Z(u)$ is the $\vdiv$-free part of $-X \vd u$.  If we consider a coordinate system around a point $x$ such that $\p_{x_j} g_{kl}(x)=0$ for all $j,k,l$ and write $\vd u = (\partial^k u) \p_{x_k}$ as in \cite[Appendix A]{PSU_hd}, then at $x$ 
\[X \vd u=v^{j}\partial_{x_{j}} (\partial^k u) \p_{x_k} = v^j \p^k (\p_{x_j} u) \p_{x_k} = v^j \vd (\p_{x_j} u). \]
Hence if we think of each $v^{j}$ as 1-form it is enough to analyze the vertical Fourier decomposition of
$A\vd w$, where $A$ is a scalar 1-form and $w = \p_{x_j} u \in \Omega_{m}$. This is precisely the content of Lemma \ref{lemma_localization_newproof} below, and combining the statement of that lemma with \eqref{eq:XV} we see that
$Z(u)=-B(u)$ where $B$ is the operator in Lemma \ref{lemma_localization_newproof} for $X\vd$. Since $B$ localizes in frequency, the lemma is proved.
\end{proof}

It remains to prove the following frequency localization statement.

\begin{Lemma} \label{lemma_localization_newproof}
Let $d = \text{\rm dim}\,M \geq 3$ and let $A \in \Omega_1$. For any $u_m \in \Omega_m$ one has 
\[
A\vd u_m = \vd \alpha(u_m) + B(u_m)
\]
where $B(u_m)$ is the $\vdiv$-free part of $A \vd u_m$. The map $B: C^{\infty}(SM) \to \mathcal{Z}$ satisfies 
\[
(B(u_m), B(w_l)) = 0, \qquad m \neq l,
\]
for any $u_m \in \Omega_m$ and $w_l \in \Omega_l$.
\end{Lemma}

The proof uses the following lemma, which follows either by relating the Hodge and connection Laplacians via a Weitzenbock identity \cite[Theorem 50]{Petersen} or by a direct computation in normal coordinates.

\begin{Lemma} \label{lemma_wedgeproduct_hodgelaplacian_new}
If $u$ and $v$ are $1$-forms on a Riemannian manifold $(M,g)$ and if $\Delta = d\delta + \delta d$ is the Hodge Laplacian, then 
\begin{align*}
%\Delta( \langle u, v \rangle ) &= \langle \Delta u, v \rangle - 2 \langle \nabla u, \nabla v \rangle + \langle u, \Delta v \rangle - 2 \,\mathrm{Ric}(u^{\sharp},v^{\sharp}), \\
\Delta(u \wedge v) &= (\Delta u) \wedge v - 2 \sum_{j=1}^{\mathrm{dim}(M)} \nabla_{E_j} u \wedge \nabla_{E_j} v + u \wedge (\Delta v) + 2 R(u^{\sharp}, v^{\sharp}, \,\cdot\,, \,\cdot\,)
\end{align*}
where $\{ E_1, \ldots, E_n \}$ is any local orthonormal frame, $\nabla$ is the Levi-Civita connection and $R$ is the Riemann curvature tensor. %Moreover, if $f$ is a $0$-form and $u$ is a $k$-form, then 
%\[
%\Delta (fu) = (\Delta f) u - 2 \nabla_{df^{\sharp}} u + f \Delta u.
%\]
\end{Lemma}

\begin{proof}[Proof of Lemma \ref{lemma_localization_newproof}]
This is a purely vertical statement and it is enough to argue for fixed $x \in M$. Moreover, since $S_x M$ is isometric to the standard sphere $S^{d-1}$, it is enough to prove the following statement: if $A \in \Omega_1$ and $u \in \Omega_m$, where $\Omega_k$ is the space of spherical harmonics of degree $k$ on $S^{d-1}$, then 
\begin{equation} \label{adum_formula}
A \,du = d(\alpha(u)) + \delta(\beta(u))
\end{equation}
where $\beta$ maps $\Omega_m$ into the set of exact $2$-forms on $S^{d-1}$ and satisfies 
\begin{equation} \label{betaum_formula}
\Delta(\beta(u)) = (\lambda_m + d-3) \beta(u) + h
\end{equation}
where $h$ is a harmonic $2$-form (hence $h=0$ for $d \geq 4$). Here $d$, $\delta$ and $\Delta = d\delta + \delta d$ are the exterior derivative, codifferential and Hodge Laplacian on $S^{d-1}$, respectively. If \eqref{adum_formula} and \eqref{betaum_formula} have been proved, then we have 
\begin{align*}
(\delta \beta(u_m), \delta \beta(w_l)) &= (\Delta \beta(u_m), \beta(w_l)) = \frac{\lambda_m+d-3}{\lambda_l+d-3} (\beta(u_m), \Delta \beta(w_l)) \\
 &= \frac{\lambda_m+d-3}{\lambda_l+d-3} (\delta \beta(u_m), \delta \beta(w_l)).
\end{align*}
Thus $B(u_m) = \delta \beta(u_m)$ satisfies $(B(u_m), B(w_l)) = 0$ for $m \neq l$ as required.

The formula \eqref{adum_formula} follows directly from the Hodge decomposition on $S^{d-1}$, and $\beta(u)$ is an exact $2$-form  (recall that there are no harmonic $1$-forms for $d \geq 3$). By taking the exterior derivative of \eqref{adum_formula}, we also see that 
\[
\Delta(\beta(u)) = dA \wedge du.
\]
Let $u \in \Omega_m$. We claim that 
\begin{equation} \label{delta_da_wedge_du}
\Delta(dA \wedge du) = (\lambda_m + d - 3) dA \wedge du.
\end{equation}
If \eqref{delta_da_wedge_du} holds, then $\Delta(\Delta(\beta(u)) - (\lambda_m + d - 3) \beta(u)) = 0$, proving \eqref{betaum_formula}.

It remains to prove \eqref{delta_da_wedge_du}. By Lemma \ref{lemma_wedgeproduct_hodgelaplacian_new}, for any local orthonormal frame $\{ E_j \}$ of $T(S^{d-1})$ we have 
\begin{align*}
\Delta(dA \wedge du) &=  (\Delta dA) \wedge du - 2 \sum_{j=1}^{d-1} \nabla_{E_j} dA \wedge \nabla_{E_j} du + dA \wedge (\Delta du) + 2 R(dA^{\sharp}, du^{\sharp}, \,\cdot\,, \,\cdot\,) \\
 &= (\lambda_m + d - 1) dA \wedge du - 2 \sum_{j=1}^{d-1} \nabla_{E_j} dA \wedge \nabla_{E_j} du + 2 R(dA^{\sharp}, du^{\sharp}, \,\cdot\,, \,\cdot\,)
\end{align*}
using that $A \in \Omega_1$, $u \in \Omega_m$ and that $\Delta$ commutes with $d$. Now if $v \in S^{d-1}$ and $w \in T_v S^{d-1}$ with $\abs{w} = 1$, and if $\gamma(t)$ is the geodesic on $S^{d-1}$ with $\dot{\gamma}(0) = w$, one has 
\[
\nabla dA|_v(w, w) = \frac{d^2}{dt^2} A(\gamma(t)) \Big|_{t=0} = \frac{d^2}{dt^2} (a_j \gamma^j(t)) \Big|_{t=0} = -A(v)
\]
using that geodesics are great circles. Thus $\nabla_{E_j} dA|_v(w) = -A(v) \langle E_j, w \rangle$, which gives that $\nabla_{E_j} dA|_v = -A(v) E_j^{\flat}$ and 
\[
\sum_{j=1}^{d-1} \nabla_{E_j} dA \wedge \nabla_{E_j} du = -A(v) \sum_{j=1}^{d-1} E_j^{\flat} \wedge \nabla_{E_j} du = -A(v) d(du) = 0.
\]
Also, on the sphere we have $R(u^{\sharp}, v^{\sharp}, \,\cdot\,, \,\cdot\,) = -u \wedge v$. These facts imply \eqref{delta_da_wedge_du}.
\end{proof}

\section{Stability for the transport equation}

In this section we will prove the main stability estimate for solutions of the transport equation $Xu = f$ in $SM$ when $f$ has finite degree. In the next section we will give the more standard form where the solenoidal part of $f$ is estimated in terms of $I_m f$.

\begin{Theorem} \label{thm_stability_transport}
Let $(M,g)$ be a compact Riemannian manifold with smooth boundary and
sectional curvature $\leq 0$, let $u \in C^{\infty}(SM)$, and write $f := Xu$. Suppose that $f$ has finite degree $m$. If $m=0$, then 
\[
\norm{f}_{L^2(SM)} \leq \norm{u}_{H^{1/2}_T(\p SM)}
\]
whereas if $m \geq 1$, then
\[
\norm{f-X(u_0 + \ldots + u_{m-1})}_{L^2(SM)} \leq \norm{u}_{H^{1/2}_T(\p SM)}.
\]
\end{Theorem}

\subsection{Shifted Pestov identity with boundary terms}

To prove Theorem \ref{thm_stability_transport} we first assume that $m \geq 1$, and discuss the case $m=0$ later. We will try to estimate $f$ in terms of $u|_{\partial SM}$ in suitable norms. The starting point is the identity from Proposition \ref{prop_pestov_localized} with $l \geq 1$:
\[
\alpha_{l-1} \norm{X_- u_l}^2 - (R \vd u_l, \vd u_l) + \norm{Z(u_l)}^2+P(u_l,u_l) = \beta_{l+1} \norm{X_+ u_l}^2.
\]
Since we are assuming non-positive sectional curvature, we have 
\[- (R \vd u_l, \vd u_l) + \norm{Z(u_l)}^2\geq 0\]
and thus
\[
\alpha_{l-1} \norm{X_- u_l}^2 +P(u_l,u_l) \leq  \beta_{l+1} \norm{X_+ u_l}^2.
\]
We divide this estimate by $\alpha_{l-1}$ (always different from zero since $l \geq 1$), which corresponds to shifting the estimate down by one half vertical derivatives since $\alpha_{l-1} \sim l$. It follows that 
\[
\norm{X_- u_l}^2 +\frac{1}{\alpha_{l-1}}P(u_l,u_l) \leq  \frac{\beta_{l+1}}{\alpha_{l-1}} \norm{X_+ u_l}^2.
\]
The constant $\frac{\beta_{l+1}}{\alpha_{l-1}}$ is exactly $D_d(l)^2$ where $D_d(l)$ is as in \cite[Lemma 5.1]{PSU_hd}. Note that $D_d(l) \leq 1$ for $d\geq 4$ and in the remaining cases it is sufficiently close to one for all practical purposes (when reading the proof it may be helpful to think that $D_d(l) \equiv 1$). %To simply the exposition in what follows we assume it is $\leq 1$.

Thus we have the following inequality:
\begin{equation}
\norm{X_- u_l}^2 +\frac{1}{\alpha_{l-1}}P(u_l,u_l) \leq D_d(l)^2 \norm{X_+ u_l}^2.
\label{eq:GKboundary}
\end{equation}
For $l \geq m$ we have $X_{-}u_{l+2}+X_{+}u_{l}=0$ and using \eqref{eq:GKboundary} we may write
\[\norm{X_- u_l}^2 +\frac{1}{\alpha_{l-1}}P(u_l,u_l) \leq  D_d(l)^2  \norm{X_- u_{l+2}}^2.\]
Starting at $l=m$ and iterating this inequality $N$ times leads to 
\[
\norm{X_- u_m}^2 \leq \left[ \prod_{j=0}^{N-1} D_d(m+2j)^2 \right] \norm{X_- u_{m+2N}}^2 - \sum_{j=0}^{N-1} \frac{\prod_{k=0}^{j-1} D_d(m+2k)^2}{\alpha_{m-1+2j}} P(u_{m+2j}, u_{m+2j})
\]
Write $\gamma_{d,m,j} = \prod_{k=0}^{j-1} D_d(m+2k)^2$ and $\gamma_{d,m,0} = 1$. In the notation of \cite[Theorem 1.1]{PSU_hd} one has $\gamma_{d,m,j} = \prod_{k=0}^{j-1} C_d(m-1+2k)^2$, and thus $\gamma_{d,m,j} \leq c_d$ where 
\begin{equation} \label{cd_estimate}
c_d = \left\{ \begin{array}{cl} 2, & d=2, \\ 1.28, & d=3, \\ 1, & d \geq 4. \end{array} \right.
\end{equation}
Since $\norm{X_- u_l}^2 \to 0$ as $l \to \infty$, we may take the limit as $N \to \infty$ to obtain 
\begin{equation} \label{xminusum_estimate}
\norm{X_- u_m}^2 \leq - \sum_{j=0}^{\infty} \frac{\gamma_{d,m,j}}{\alpha_{m-1+2j}} P(u_{m+2j}, u_{m+2j}).
\end{equation}
The argument above gives a completely analogous inequality for $\norm{X_- u_{m+1}}^2$, and adding these two inequalities leads to 
\begin{equation} \label{xminusum_estimate_second}
\norm{X_- u_m}^2 + \norm{X_- u_{m+1}}^2 \leq - \sum_{k=0}^{\infty} b_{m,k} P(u_{m+k}, u_{m+k})
\end{equation}
where 
\[
b_{m,k} = \left\{ \begin{array}{cl} \frac{\gamma_{d,m,j}}{\alpha_{m-1+2j}}, & k = 2j, \\[5pt] \frac{\gamma_{d,m+1,j}}{\alpha_{m+2j}}, & k = 2j+1. \end{array} \right.
\]
Define $r:=u_0 + u_{1}+\cdots+u_{m-1}$. Then the transport equation $Xu=f$ also gives
\begin{equation} \label{f_xq_decomposition}
Xr+X_- u_m + X_{-}u_{m+1}=f
\end{equation}
and thus $\norm{f-Xr}^2 = \norm{X_- u_m}^2 + \norm{X_- u_{m+1}}^2$. This yields
\begin{equation}
\norm{f-Xr}^2 \leq - \sum_{k=0}^{\infty} b_{m,k} P(u_{m+k}, u_{m+k}).
\label{eq:fminusq}
\end{equation}

If we assume $m=0$, then the equation $Xu = f$ implies $X_- u_1 = f$, and \eqref{xminusum_estimate} gives 
\begin{equation}
\norm{f}^2 \leq - \sum_{j=0}^{\infty} b_{1,2j} P(u_{1+2j}, u_{1+2j}) \leq - \sum_{k=0}^{\infty} b_{1,k} P(u_{1+k}, u_{1+k}).
\label{eq:fmzero}
\end{equation}
Thus to prove Theorem \ref{thm_stability_transport} for $m \geq 0$, it remains to bound the right hand side of \eqref{eq:fminusq} for all $m \geq 1$.

\subsection{The right hand side of \eqref{eq:fminusq}: the space $H^{1/2}_{T}(\partial SM)$.} Let $m \geq 1$. Motivated by \eqref{eq:fminusq}
and the fact that $P$ is defined in terms of the (horizontal) tangential operator $T$, we define a natural
$H^{1/2}$-space as follows. Define the $H^{1}_{T}(\partial SM)$-norm by setting
\[\norm{u}_{H^{1}_{T}}^{2}:=\norm{u}^{2}_{L^{2}}+\norm{\ehdp u}_{L^{2}}^{2}.\]
The space $H^{1/2}_{T}(\partial SM)$ is defined by complex interpolation between $L^2$ and $H^{1}_{T}$.
The norm $H^{-1/2}_{T}$ is defined by duality, and then $H^{-1/2}_T$ is also the interpolation space between $L^2$ and $H^{-1}_T$ (see \cite[Corollary 4.5.2]{BerghLofstrom}).

\begin{Remark}{\rm Note that from \eqref{eq:Tfull} we have
\begin{equation} \label{T_ltwo_bound}
\norm{Tu}_{L^{2}} \leq \norm{\ehdp u}_{L^{2}}
\end{equation}
since $\abs{Tu}^2 = \mu^2 \abs{\ehdp u}^2 + \abs{X^{\parallel} u}^2 \leq (\mu^2 + \abs{v^{\parallel}}^2) \abs{\ehdp u}^2 = \abs{v}^2 \abs{\ehdp u}^2 = \abs{\ehdp u}^2$.

\label{rem:Tbound}

}
\end{Remark}

Now we use the key property of localization given by Lemma \ref{lemma_curvature_fourier_localization}
to observe that
\[\sum_{k=0}^{\infty} b_{m,k} P(u_{m+k}, u_{m+k})=P\left(\sum_{k=0}^{\infty} u_{m+k}, \sum_{l=0}^{\infty} b_{m,l} u_{m+l} \right).\]
We define an operator $B = B_m:C^{\infty}(\partial SM)\to C^{\infty}(\partial SM)$ by setting
\[Bu:=\sum_{l=0}^{\infty} b_{m,l} u_{m+l}.\]
Since $m \geq 1$, the constant $b_{m,l}$ is well defined also when $l=0$. Now \eqref{eq:fminusq} becomes 
\begin{equation} \label{fminusq_second}
\norm{f-Xr}^2 \leq -P(u,Bu) = -(Tu, \vd Bu).
\end{equation}
Here is the main claim:

\begin{Lemma} \label{lemma:h1/2}
Given $u\in C^{\infty}(\partial SM)$ we have
\[|(Tu,\vd Bu)| \leq \norm{u}^{2}_{H^{1/2}_{T}}.\]
\end{Lemma}

\begin{proof} We may write
\[(Tu,\vd Bu)=-(B\vdiv T u,u).\]
By the definitions, it suffices to show that
\[\norm{B\vdiv Tu}_{H^{-1/2}_{T}} \leq \norm{u}_{H^{1/2}_{T}}.\]
By interpolation, this follows from the next two inequalities
\begin{equation}
\norm{B\vdiv Tu}_{L^2} \leq \norm{u}_{H^{1}_{T}},
\label{eq:01}
\end{equation}
\begin{equation}
\norm{B\vdiv Tu}_{H^{-1}_{T}} \leq \norm{u}_{L^2}.
\label{eq:-10}
\end{equation}

To prove these estimates we first establish the property
\begin{equation}
\norm{\vd Bu}_{L^{2}} \leq \norm{u}_{L^{2}}.
\label{eq:zeroorder}
\end{equation}
Indeed using the definition of $B$,
\begin{align*}
\norm{\vd Bu}_{L^{2}}^{2}&=(Bu,\Delta Bu)=\sum_{l=0}^{\infty} \lambda_{m+l} b_{m,l}^2 \norm{u_{m+l}}^{2}_{L^{2}}.
\end{align*}
To prove \eqref{eq:zeroorder} we will show that $\lambda_{m+l} b_{m,l}^2 \leq 1$ for $m \geq 1$ and $l \geq 0$. If $m=1$ and $l=0$, then $\lambda_1 b_{1,0}^2 = \frac{\lambda_1}{\alpha_0^2} = \frac{1}{d-1} \leq 1$, so we may assume $m, l \geq 1$. One has $\gamma_{d,m,j} \leq c_d$, which gives $\lambda_{m+l} b_{m,l}^2 \leq c_d^2 \frac{\lambda_{m+l}}{\alpha_{m-1+l}^2}$. Observe that simplifying \eqref{alphal_formula} gives $\alpha_l = \frac{(2l+d-2)(l+d-1)}{l+d-2}$ for all $l \geq 1$. We thus have, writing $k=m+l \geq 2$,  
\begin{align*}
\frac{\lambda_k}{\alpha_{k-1}^2} &= \frac{k(k+d-2)(k+d-3)^2}{(2k+d-4)^2(k+d-2)^2} \\
 &= \frac{1}{4} \frac{k^3 + (2d-6)k^2 + (d-3)^2 k}{(k^2 + (d-4)k + (\frac{d-4}{2})^2)(k+d-2)} \\
 &= \frac{1}{4} \frac{k^3 + (2d-6)k^2 + (d^2 - 6d + 9) k}{k^3 + (2d-6)k^2 + (d^2 - 6d + 8 + (\frac{d-4}{2})^2)k + (d-2) (\frac{d-4}{2})^2}.
\end{align*}
Thus if $d = 2$ or $d \geq 6$, one has $(\frac{d-4}{2})^2 \geq 1$ and hence $\frac{\lambda_k}{\alpha_{k-1}^2} \leq \frac{1}{4}$. If $d=3,4,5$ we estimate 
\begin{align*}
\frac{\lambda_k}{\alpha_{k-1}^2} &\leq \frac{1}{4} \frac{k^3 + (2d-6)k^2 + (d^2 - 6d + 9) k}{k^3 + (2d-6)k^2 + (d^2 - 6d + 8 )k} \leq \frac{1}{4} \left[ 1 + \frac{1}{(k+d-3)^2-1} \right] \leq \frac{1}{3}
\end{align*}
using $k \geq 2$. Combining these estimates with \eqref{cd_estimate}, we have 
\[
\lambda_{m+l} b_{m,l}^2 \leq c_d^2 \frac{\lambda_{m+l}}{\alpha_{m-1+l}^2} \leq 1.
\]
The estimate \eqref{eq:zeroorder} follows. Since $-\vd B$ is the adjoint of $B\vdiv$, using \eqref{eq:zeroorder} and \eqref{T_ltwo_bound} yields
\[\norm{B\vdiv Tu}_{L^2} \leq \norm{Tu}_{L^{2}} \leq \norm{u}_{H^{1}_{T}}\]
thus proving \eqref{eq:01}.

Finally to prove \eqref{eq:-10}, we note that $B\vdiv T = \vdiv TB$ by Lemma \ref{lemma_curvature_fourier_localization}. Using Lemma \ref{lemma:Qstar} we may write
\begin{align*}
\norm{B\vdiv Tu}_{H^{-1}_{T}}&=\sup_{\norm{h}_{H^{1}_{T}}=1}(\vdiv TBu,h)\\
&=\sup_{\norm{h}_{H^{1}_{T}}=1}-(T^*\vd Bu,h)\\
&=\sup_{\norm{h}_{H^{1}_{T}}=1}-(\vd Bu,Th)\\
&\leq \sup_{\norm{h}_{H^{1}_{T}}=1}\norm{u}_{L^{2}}\norm{Th}_{L^{2}}\\
&\leq \norm{u}_{L^{2}},
\end{align*}
where in the penultimate line we used \eqref{eq:zeroorder} and \eqref{T_ltwo_bound}.
\end{proof}

Theorem \ref{thm_stability_transport} for $m \geq 1$ now follows from \eqref{fminusq_second} and Lemma \ref{lemma:h1/2}. When $m=0$, it follows from \eqref{eq:fmzero} and Lemma \ref{lemma:h1/2}.

\section{Stability for the solenoidal part}

We now rewrite Theorem \ref{thm_stability_transport} in terms of the solenoidal part of the tensor inducing $f$ and extend the result to $H^1$ regularity. Recall that the map
$$\ell_m: C^{\infty}(S^m(T^*M))\to \bigoplus_{k=0}^{[m/2]}\Omega_{m-2k},$$ 
is an isomorphism and it gives a natural identification between functions in $\Omega_m$ and {\it trace-free} symmetric $m$-tensors (for details on this see \cite{GK2,DS,PSU_hd}). The identification actually holds pointwise for every $x\in M$. Moreover, the $L^2$-norms on trace free symmetric $m$-tensors and functions in $\Omega_{m}$ are the same up to a constant depending only on $d$ and $m$ (cf.\ \cite[Lemma 2.4]{DS}).

Given a tensor $\tilde{f}\in H^{1}(S^m(T^*M))$, we let $f=\ell_{m}\tilde{f}$ (using $\ell_{m}$ on $H^1$ spaces). Note if $m$ is even (resp. odd), then $f$ has zero odd (resp. even) Fourier modes. If $Xu=f$, we let 
$$q:=\sum_{k=0}^{[(m-1)/2]}u_{m-2k-1}.$$
Upon defining $\tilde{q}:=\ell_{m-1}^{-1}q$, the well-known relation $X\ell_{m-1}=\ell_{m}d^{s}$ implies that $\ell_{m}^{-1}(f-Xq)=\tilde{f}-d^{s}\tilde{q}$. To simplify the notation we shall the drop 
the tildes, identify $f$ with $\tilde{f}$, $q$ with $\tilde{q}$ and use that the $L^2$-norms are equivalent.

We first collect regularity properties of solutions of transport equations involving $H^1$ tensor fields.

\begin{Lemma} \label{lemma_uf_regularity}
Let $(M,g)$ be a compact simple manifold. Given $f \in H^1(S^m(T^* M))$, there is $u^f \in H^1(SM)$ satisfying 
\begin{equation} \label{uf_hone_properties}
\text{$Xu^f = -f$ in $SM$, \qquad $u^f|_{\p_- SM} = 0$, \qquad $u^f|_{\p_+ SM} = If$.}
\end{equation}
Moreover, one has $u^f|_{\p SM} \in H^1(\p SM)$ and $If \in H^1_0(\p_+ SM)$.
\end{Lemma}
\begin{proof}
If $f \in C^{\infty}(S^m(T^*M))$, define $u^f$ on $SM$ by 
\[
u^f(x,v) = \int_0^{\tau(x,v)} f(\varphi_t(x,v)) \,dt
\]
where $\varphi_t$ is the geodesic flow on $SM$. One has $u^f \in C^{\infty}(SM \setminus S(\p M)) \cap C(SM)$ since the same is true for $\tau$, and \eqref{uf_hone_properties} holds for $u^f$. By \cite[Corollary 1]{Sh95}, the map $f \mapsto u^f$ extends as a bounded map $H^1(S^m(T^*M)) \to H^1(SM)$. (This boils down to the fact that $\vd \tau$ and $\ehdp \tau$, where the operator $\ehdp$ is extended smoothly to $SM$, are uniformly bounded on $SM \setminus S(\p M)$, see \cite[Lemma 4.1.3]{Sh} and \cite[Lemma 5.1]{DPSU07}.) Moreover, by \cite[Theorem 4.2.1]{Sh} the map $f \mapsto If$ extends as a bounded map $H^1(S^m(T^*M)) \to H^1(\p_+ SM)$. Then the properties \eqref{uf_hone_properties} remain valid for $f \in H^1$ (the boundary value of $u^f$ is in $H^{1/2}(\p SM)$ by the trace theorem). Since $If$ vanishes on the boundary of $\p_+ SM$ when $f \in C^{\infty}$, one has $If \in H^1_0(\p_+ SM)$ first for $f \in C^{\infty}$ and then for $f \in H^1$ by density. Since $u^f|_{\p SM} = E_0 (If)$ where $E_{0}$ denotes extension by zero from $\p_+ SM$ to $\p SM$, we have $u^f|_{\p SM} \in H^1(\p SM)$ when $f \in H^1$.
\end{proof}

Next we give a version of Theorem \ref{thm_stability_transport} for $H^1$ tensor fields. In its statement $q$ is determined by $u^f$ as described above.

\begin{Theorem} \label{thm_stability_transport_second}
Let $(M,g)$ be a simply connected compact manifold with strictly convex boundary and
sectional curvature $\leq 0$, and let $f \in H^1(S^m(T^*M))$. If $m=0$, then 
\begin{equation} \label{f_tensor_initial_estimate}
\norm{f}_{L^2(M)} \leq C \norm{u^f}_{H^{1/2}_T(\p SM)}
\end{equation}
whereas if $m \geq 1$, then
\begin{equation} \label{fminusdsq}
\norm{f-d^s q}_{L^2(M)} \leq C \norm{u^f}_{H^{1/2}_T(\p SM)}.
\end{equation}
Here $C$ only depends on $d$ and $m$.
\end{Theorem}
\begin{proof}
Let $m \geq 1$ (the case $m=0$ is analogous). Going back to \eqref{xminusum_estimate} and using Lemma \ref{lemma:h1/2}, one has the inequality 
\[
\norm{X_- u_m}^2 + \norm{X_- u_{m+1}}^2 \leq \norm{u}_{H^{1/2}_T(\p SM)}^2, \qquad u \in C^{\infty}(SM).
\]
Since functions in $H^1(SM)$ have traces in $H^{1/2}(\p SM)$, and hence also in $H^{1/2}_T(\p SM)$, the above inequality holds for $u \in H^1(SM)$ by density. Then it is enough to take $u = u^f$, where $u^f \in H^1(SM)$ by Lemma \ref{lemma_uf_regularity}, and to note that by \eqref{f_xq_decomposition} and by equivalence of the $L^2$ norms 
\[
\norm{X_- u_m}^2 + \norm{X_- u_{m+1}}^2 = \norm{f-Xr}_{L^2(SM)}^2 \geq c(d,m) \norm{f-d^s q}_{L^2(M)}^2. \qedhere
\]
\end{proof}

The estimate \eqref{f_tensor_initial_estimate} for $m=0$ is already in the form that we want, so we will focus on the case $m \geq 1$. Using the potential and solenoidal decomposition, we may write
$f=f_{s}+d^s p$ where $\delta^{s}f_{s}=0$ and $p$ is an $(m-1)$-tensor such that $p|_{\partial M}=0$.
Let $w=p-q$. Then integrating by parts
\begin{align}
\norm{f-d^{s}q}^{2}&=\norm{f_{s}+d^{s} w}^{2} \label{fminusdsq_estimate} \\
&=\norm{f_{s}}^{2}+2(f_{s},d^{s}w)+\norm{d^{s}w}^{2} \notag \\
&=\norm{f_{s}}^{2}+2(\iota_{\nu}f_{s},w)_{\partial M}+\norm{d^{s}w}^{2} \notag \\
&\geq \norm{f_{s}}^{2}-2 \abs{(\iota_{\nu}f_{s},q)_{\partial M}}. \notag
\end{align}
Next we observe that for any $\varepsilon>0$
\begin{equation} \label{iotanufs_estimate}
2 \abs{(\iota_{\nu}f_{s},q)_{\partial M}} \leq \frac{1}{\varepsilon}\norm{q}^{2}_{H^{1/2}(\partial M)}+\varepsilon\norm{\iota_{\nu}f_{s}}^{2}_{H^{-1/2}(\partial M)}.
\end{equation}

We now claim:

\begin{Lemma} \label{lemma:parts}
We have
\[\norm{\iota_{\nu}f_{s}}_{H^{-1/2}(\partial M)}\lesssim \norm{f_{s}}_{L^{2}(M)}.\]
\end{Lemma}
\begin{proof}
This is a duality argument, but it is important that $\delta^{s}f_{s}=0$. Consider a bounded extension map for symmetric $(m-1)$-tensors, $e:H^{1/2}(\partial M)\to H^{1}(M)$ (such a map can be constructed from a corresponding extension map for functions by working in local coordinates and using a partition of unity). Now write
\begin{align*}
\norm{\iota_{\nu}f_{s}}_{H^{-1/2}(\partial M)}&=\sup_{\norm{h}_{H^{1/2}(\partial M)}=1}\int_{\partial M}\langle \iota_{\nu}f_{s},h\rangle\,dS\\
&=\sup_{\norm{h}_{H^{1/2}(\partial M)}=1}(-(\delta^{s}f_{s},e(h))+(f_{s},d^{s}e(h)))\\
&=\sup_{\norm{h}_{H^{1/2}(\partial M)}=1}(f_{s},d^{s}e(h))\\
&\lesssim \norm{f_{s}}_{L^2(M)}. \qedhere
\end{align*}
\end{proof}

Combining \eqref{fminusdsq}--\eqref{iotanufs_estimate} with Lemma \ref{lemma:parts} and choosing $\eps$ small enough, it follows that 
\begin{equation} \label{fs_initial_estimate}
\norm{f_s}_{L^2(M)}^2 \lesssim \norm{u^f}_{H^{1/2}_T(\p SM)}^2 + \norm{q}_{H^{1/2}(\partial M)}^2.
\end{equation}
The next two lemmas will be useful when estimating the last term on the right.

\begin{Lemma} \label{lemma:aux}
Given $m\geq 0$, there is a constant $C>0$ such that for any tensor $q$ of order $m$
\[\norm{q}_{H^{1/2}(\partial M)}\leq C \norm{\ell_{m} q}_{H_{T}^{1/2}(\partial SM)}.\]
\end{Lemma}
\begin{proof} Recall that we identify symmetric $m$-tensors with functions in $\bigoplus_{k=0}^{[m/2]}\Omega_{m-2k}$ via $\ell_m$ as explained at the beginning of this section. By interpolation, it is enough to show that $\norm{\ell_{m}^{-1} h}_{L^2}\lesssim \norm{h}_{L^2}$ and 
$\norm{\ell_{m}^{-1} h}_{H^1}\lesssim \norm{h}_{H^{1}_{T}}$ when $h \in \bigoplus_{k=0}^{[m/2]}\Omega_{m-2k}$.  The first inequality follows from the equivalence of the $L^{2}$-norms. For the second inequality, observe that locally a symmetric $m$-tensor field can be written as $q=q_{i_{1}\dots i_{m}}dx^{i_{1}}\otimes\dots\otimes dx^{i_{m}}$. The $H^{1}$-norm of $q$ in $\partial M$ consists of the $L^2$-norm of $q$ plus the $L^{2}$ norm
of the components $q_{i_{1}\dots i_{m}}(x)$ {\it tangentially} to $M$. Locally $\ell_{m}q$ has the form
$q_{i_{1}\dots i_{m}}v^{i_{1}}\dots v^{i_{m}}$. When we apply $\ehdp$ to $\ell_{m}q$ all the tangential derivatives in the direction of $\partial M$ will appear. There will also be some vertical derivatives (involving the Christoffel symbols), but since $\ell_{m}q$ is a polynomial of degree $m$ in $v$, these terms can all be controlled by the $L^{2}$-norm of $\ell_{m}q$.
Thus $\norm{q}_{H^1}\lesssim \norm{\ell_{m} q}_{H^{1}_{T}}$ follows, and this may be rewritten as $\norm{\ell_{m}^{-1} h}_{H^1}\lesssim \norm{h}_{H^{1}_{T}}$.
\end{proof}

\begin{Lemma}[The $H^1_T(\p SM)$ norm localizes in frequency] One has 
\[
\norm{u}_{H^1_T(\partial SM)}^2 = \sum_{m=0}^{\infty} \norm{u_m}_{H^1_T(\partial SM)}^2
\]
for all $u\in H^1_T(\partial SM)$. In particular, $ \norm{\sum_{l=0}^m u_l}_{H^1_T(\partial SM)} \leq \norm{u}_{H^1_T(\partial SM)}$ when $m \geq 0$.
\label{lemma:fbound}
\end{Lemma}
\begin{proof} The proof is somewhat indirect and is based on the following observations.
\begin{enumerate}
\item Let $W$ be a vector field on $M$ and let $\mathbb{W}=(W,0)$ be its horizontal lift to $SM$. Then
$\Delta \mathbb{W}=\mathbb{W}\Delta$ where $\Delta$ is the vertical Laplacian. This can be seen by taking a geodesic coordinate neighbourhood around a point $x$, so that $\p_j g_{kl}(x) = 0$ for all $j,k.l$. In that case if we write $W=w^{i}(x)\partial_{x_{i}}$, then
$(\mathbb{W}u)(x,v)=w^{i}(x)\frac{\partial u}{\partial x_{i}}$ and thus $\mathbb{W}:\Omega_{m}\to\Omega_{m}$. (Another way to prove this is to check that $[\delta_j, \Delta] = 0$, using the notation and commutator formulas in \cite[Appendix A]{PSU_hd}.)
\item There is a neighbourhood $U_{\varepsilon}$ of $\partial M$ in $M$ diffeomorphic to $\partial M\times [0,\varepsilon)$
via $\partial M\times [0,\varepsilon)\ni (x,t)\mapsto \exp_{x}(-t\nu(x))\in U_{\varepsilon}$. This allows us to naturally extend to $U_{\varepsilon}$ the exterior unit normal $\nu$ to a vector field, still denoted by $\nu$.
\item A smooth function $u\in C^{\infty}(\partial SM)$ can be extended to a smooth function $u^{\bbnu}\in C^{\infty}(SU_{\varepsilon})$ simply by making it constant on the orbits of the flow of $\bbnu$, the horizontal lift of $\nu$.
By item (1) we have
\begin{equation}
(u_{m})^{\bbnu}=({u}^{\bbnu})_{m}
\label{eq:extcomm}
\end{equation}
\item 
Let $f_t$ be the flow of $\bbnu$ in $SM$, and let $V_{\eps}$ be the neighbourhood of $\p SM$ in $SM$ diffeomorphic to $\p SM \times [0,\eps)$ via $(x,v,t) \mapsto f_{-t}(x,v)$. Since $f_{-t}(x,v) = (x(t), v(t))$ where $x(t) = \exp_x(-t\nu)$ is the normal geodesic and $v(t)$ is the parallel transport of $v$ along $x(t)$, one has $V_{\eps} = SU_{\eps}$ (the map $v \mapsto v(t)$ is bijective from $S_x M$ onto $S_{x(t)} M$).
\end{enumerate}

Let $u \in C^{\infty}(\p SM)$. The fact that $\bbnu (u^{\bbnu}) = 0$ implies that $\ehdp u=(\ehd u^{\bbnu})|_{\partial SM}$ and thus for $(x,v)\in \partial SM$
we have
\[\lim_{\varepsilon\to 0}\frac{1}{\varepsilon}\int_{0}^{\varepsilon} |\ehd u^{\bbnu}(x,v,t)|^{2}\,dt=|\ehdp u(x,v)|^{2}.\]
Integrating over $\p SM$ and using that $V_{\eps} = S U_{\eps}$ by (4), this gives
\begin{equation}
\lim_{\varepsilon\to 0}\frac{1}{\varepsilon}\norm{\ehd u^{\bbnu}}^{2}_{L^{2}(SU_{\varepsilon})}=\norm{\ehdp u}^{2}_{L^{2}(\partial SM)}.
\label{eq:limit}
\end{equation}

% consider the function
%\[y(t):= |\ehd u^{\bbnu}(x,v,t)|^{2}.\]
%Using that $\bbnu u^{\bbnu}=0$ we deduce
%\[\dot{y}=2\langle \bbnu \ehd u^{\bbnu},\ehd u^{\bbnu}\rangle=2\langle [\bbnu, \ehd]u^{\bbnu}, \ehd u^{\bbnu}\rangle.\]
%Using that $ [\bbnu, \ehd]$ is given in terms of curvature and $\vd$ we obtain
%\[C_{2}(y-|\vd u^{\bbnu}|^{2})\leq \dot{y}\leq C_{1}(y+|\vd u^{\bbnu}|^{2}).\]
%Integrating we find new constants $C_1$ and $C_2$ such that
%\[C_{2}(|\ehdp u|^2-|\vd u|^{2})\leq y(t)\leq C_{1}(|\ehdp u|^2+|\vd u^{\bbnu}|^{2}).\]
%Integrating in $SU$ and using as measure $d\Sigma^{2d-2}\wedge dt$ we find
%\[\norm{\ehdp u}^{2}_{L^{2}(\partial SM)}-\norm{\vd u}^{2}_{L^2(\partial SM)}\lesssim
%\norm{\ehd (u^{\bbnu})}^2_{L^{2}(SU)}\lesssim \norm{\ehdp u}^{2}_{L^{2}(\partial SM)}+\norm{\vd u}^{2}_{L^2(\partial SM)} .\]
%The left hand side inequality implies

%\begin{equation}
%\norm{\ehdp u_{m}}^{2}_{L^{2}(\partial SM)}\lesssim \norm{\ehd (u^{\bbnu})_{m}}^{2}_{L^{2}(SU)}+\norm{u_{m}}^{2}_{L^{2}(\partial SM)}.
%\label{eq:btof}
%\end{equation}
We now recall that it is possible to write for any $w\in C^{\infty}(SU_{\varepsilon})$ (all norms in $L^{2}(SU_{\varepsilon})$), cf. \cite[proof of Lemma 5.1]{LRS18}:
\begin{equation}
 \norm{\ehd w}^2=\norm{Z(w)}^{2}+\norm{X_{-}w_{1}}^{2}+\sum_{l=1}^{\infty}A(d,l)\norm{X_{+} w_{l-1}}^{2}+B(d,l)\norm{X_{-} w_{l+1}}^{2},
\label{eq:LRS}
\end{equation}
where $A(d,l)=2+\frac{d-2}{2}$ and $B(d,l)=2+\frac{1}{d+l-2}$. By Lemma \ref{lemma_curvature_fourier_localization} one may write $\norm{Z(w)}^2 = \sum \norm{Z(w_m)}^2$, and thus using \eqref{eq:LRS} for $w=u^{\bbnu}$ and $w = (u^{\bbnu})_m$ for each $m$ we deduce
\[
\norm{\ehd u^{\bbnu}}^{2}_{L^{2}(SU_{\varepsilon})} = \sum_{m=0}^{\infty} \norm{\ehd (u^{\bbnu})_{m}}^{2}_{L^{2}(SU_{\varepsilon})}.
\]
Dividing this by $\varepsilon$ and taking the limit as $\varepsilon\to 0$, the identities \eqref{eq:extcomm} and \eqref{eq:limit}
yield 
\[
\norm{\ehdp u}^{2}_{L^{2}(\partial SM)} = \sum_{m=0}^{\infty} \norm{\ehdp u_m}^{2}_{L^{2}(\partial SM)}.
\]
This implies the desired claim for $u \in C^{\infty}(\p SM)$, and the result follows since $C^{\infty}(\p SM)$ is dense in $H^1_T(\p SM)$. (The density claim can be proved by using a partition of unity, convolution approximation in coordinate neighborhoods, and the Friedrichs lemma \cite[Lemma 17.1.5]{Hormander}.)
\end{proof}

We now put the arguments above together to derive:

\begin{Theorem} \label{thm:main0}
Let $(M,g)$ be a simply connected compact manifold with strictly convex boundary and
non-positive sectional curvature. Given $m\geq 0$ there is a constant $C>0$ such that for any $f\in H^1(S^{m}(T^*M))$, one has 
\[\norm{f_{s}}_{L^{2}(M)}\leq C\norm{u^f}_{H_{T}^{1/2}(\partial SM)}.\]
\end{Theorem}
\begin{proof} For $m=0$ this is just \eqref{f_tensor_initial_estimate}, so we assume $m \geq 1$. Combining \eqref{fs_initial_estimate} and Lemma \ref{lemma:aux} we derive 
\[\norm{f_{s}} \lesssim \norm{u^f}_{H^{1/2}_{T}} + \norm{\ell_{m-1} q}^2_{H^{1/2}_{T}}.\]
Recall that $\ell_{m-1} q=\sum_{k=0}^{ [(m-1)/2]}u_{m-2k-1}^f$. Interpolating the bound $\norm{\sum_{k=0}^{ [(m-1)/2]}u_{m-2k-1}}_{L^2(\p SM)} \leq \norm{u}_{L^2(\p SM)}$ for $u \in L^2(\p SM)$ with the bound in Lemma \ref{lemma:fbound} gives 
\[\left\|\sum_{k=0}^{ [(m-1)/2]}u_{m-2k-1}\right\|_{H^{1/2}_{T}} \leq \norm{u}_{H^{1/2}_{T}}, \qquad u \in H^{1/2}_T(\p SM), \]
and the result follows by taking $u = u^f$.
\end{proof}

We can refine this further and prove Theorem \ref{thm:main_final} in the introduction. Define the space $H_{T}^{1/2}(\partial_{+}SM)$ as the interpolation space between $H^1_T(\partial_+ SM)$ and $L^2(\partial_+ SM)$. %restriction of $H_{T}^{1/2}(\partial SM)$ to $\partial_+ SM$ with the quotient norm.

\begin{proof}[Proof of Theorem \ref{thm:main_final}]
Theorem \ref{thm:main0} gives 
\begin{equation} \label{fs_honezero_estimate}
\norm{f_{s}}_{L^{2}(M)}\leq C\norm{E_{0}(If)}_{H_{T}^{1/2}(\partial SM)}
\end{equation}
where $E_0$ is extension by zero from $\p_+ SM$ to $\p SM$. We define $H_{T,0}^{1}(\partial_{+}SM)$
as the closure of $C^{\infty}_{c}((\partial_{+}SM)^{\mathrm{int}})$ with respect to the $H^{1}_{T}$-norm, and $H_{T,0}^{1/2}(\partial_{+}SM)
$ as the interpolation space between $L^{2}$ and $H_{T,0}^{1}(\partial_{+}SM)$. Since $E_0$ is bounded $H^1_{T,0}(\partial_+ SM) \to H^1_T(\partial SM)$ and $L^2(\partial_+ SM) \to L^2(\partial SM)$, it is also bounded 
\begin{equation} \label{ezero_hhalf_estimate}
E_{0}: H_{T,0}^{1/2}(\partial_{+}SM)\to H^{1/2}_{T}(\partial SM).
\end{equation}
Now, if $f \in H^1(S^{m}(T^*M))$, then $If$ is in $H_{0}^{1}(\partial_{+}SM)$ by Lemma \ref{lemma_uf_regularity} and hence also in the larger space $H_{T,0}^{1/2}(\partial_{+}SM)$. Combining \eqref{fs_honezero_estimate} and \eqref{ezero_hhalf_estimate} proves the result.
\end{proof}

\begin{Remark} Theorem \ref{thm:main_final} remains true for $f \in H^{1/2}(S^m(T^* M))$ by density, since $I$ is bounded $H^{1/2}(S^m(T^* M)) \to H^{1/2}(\p_+ SM)$ by \cite[Theorem 4.2.1]{Sh} and interpolation. It would be interesting if one could prove Theorem \ref{thm:main_final} for $f \in L^2(S^m(T^* M))$. However, in general we do not know if $I$ is bounded $L^{2}(S^m(T^* M))\to H^{1/2}_{T}(\partial_{+}SM)$.
Also, our approach with the Pestov identity as it stands is unable to produce stability estimates for higher order Sobolev norms.
\end{Remark}

\bibliographystyle{alpha}

\begin{thebibliography}{DKSU09}



%\bibitem[AR97]{AR} Yu. Anikonov, V. Romanov, {\it On uniqueness of determination of a form of first degree by its integrals along geodesics}, J. Inverse Ill-Posed Probl. {\bf 5} (1997), 467--480.

\bibitem[AS20]{AS19} Y. Assylbekov, P. Stefanov, {\it Sharp stability estimate for the geodesic ray transform}, Inverse Problems {\bf 36} (2020) 025013.

\bibitem[BL76]{BerghLofstrom} J.\ Bergh, J.\ L\"ofstr\"om, Interpolation spaces. Springer-Verlag, 1976.

%\bibitem[Bo10]{Boman} J.\ Boman, {\it Local non-injectivity for weighted Radon transforms}, Contemp. Math. {\bf 559} (2010), 39--48.

\bibitem[BS18]{BS18} J. Boman, V. Sharafutdinov, {\it Stability estimates in tensor tomography,} Inverse Probl. Imaging {\bf 12} (2018) 1245--1262.

%\bibitem[Bo09]{Boucetta} M. Boucetta, {\it Spectra and symmetric eigentensors of the Lichnerowicz Laplacian on $S^n$,} Osaka J. Math. {\bf 46} (2009) 235--254.



%\bibitem[CS98]{CS} C. Croke, V.A. Sharafutdinov, \emph{Spectral rigidity of a compact negatively curved manifold}.
%Topology \textbf{37} (1998) 1265--1273.

%\bibitem[Da06]{Da06} N.S. Dairbekov, {\it Integral geometry for nontrapping manifolds,} Inverse Problems {\bf 22} (2006) 431--445.

\bibitem[DPSU07]{DPSU07} N. Dairbekov, G.P. Paternain, P. Stefanov, G. Uhlmann, {\it The boundary rigidity problem in the presence of a magnetic field,} Adv. Math. {\bf 216} (2007) 535--609. 



\bibitem[DS11]{DS} N.S. Dairbekov, V.A. Sharafutdinov, {\it On conformal Killing symmetric tensor fields on Riemannian manifolds,} Siberian Advances in Mathematics {\bf 21} (2011) 1--41.

\bibitem[Gu17]{Gu15} C. Guillarmou, {\it Lens rigidity for manifolds with hyperbolic trapped sets}, J. Amer. Math. Soc. {\bf 30} (2017) 561--599.

\bibitem[GMT20]{GMT} C. Guillarmou, M. Mazzucchelli, L. Tzou, {\it Boundary and lens rigidity for non-convex manifolds,} to appear in Amer. J. Math.



%\bibitem[GK80a]{GK} V. Guillemin, D. Kazhdan, {\it Some inverse spectral
%results for negatively curved 2-manifolds,} Topology {\bf 19}
%(1980) 301--312.

%\bibitem[GK80b]{GK1} V. Guillemin, D. Kazhdan, {\it On the cohomology of certain dynamical systems,} Topology {\bf 19} (1980) 291--299.

\bibitem[GK80]{GK2} V. Guillemin, D. Kazhdan, {\it Some inverse spectral results for negatively curved n-manifolds,}  Geometry of the Laplace operator (Proc. Sympos. Pure Math., Univ. Hawaii, Honolulu, Hawaii, 1979), pp. 153--180, Proc. Sympos. Pure Math., XXXVI, Amer. Math. Soc., Providence, R.I., 1980.


\bibitem[H\"o85]{Hormander}
L. H\"ormander, The analysis of linear partial differential operators, vols.\ I-IV. Springer, 1983--1985.

\bibitem[IP18]{IP18} J.Ilmavirta, G. P. Paternain, {\it Broken ray tensor tomography with one reflecting obstacle}, 
arXiv:1805.04947, to appear in Comm. Anal. Geom.


%\bibitem[Jo05]{Jost} J. Jost, {\it Riemannian geometry and geometric analysis.} 4th edition, Springer-Verlag, Berlin Heidelberg, 2005.

\bibitem[Kn02]{Kn} G. Knieper, {\it Hyperbolic dynamics and Riemannian geometry.} Handbook of dynamical systems, Vol. 1A, 453--545, North-Holland, Amsterdam, 2002.

\bibitem[LRS18]{LRS18} J.  Lehtonen, J. Railo, M. Salo, {\it Tensor tomography on Cartan-Hadamard manifolds,} Inverse Problems {\bf 34} (2018), no. 4, 044004, 27 pp.

\bibitem[Mo19]{M19} F. Monard, {\it Functional relations, sharp mapping properties and regularization
of the X-ray transform on disks of constant curvature,} to appear in SIAM J. Math. Anal.

\bibitem[MNP19]{MNP} F. Monard, R. Nickl, G.P. Paternain, {\it Efficient nonparametric Bayesian inference for X-ray transforms,} Ann. Statist. {\bf 47} (2019) 1113--1147.



%\bibitem[Mu77]{Mu} R.G. Mukhometov, \emph{The reconstruction problem of a two-dimensional Riemannian metric, and integral geometry (Russian)},
%Dokl. Akad. Nauk SSSR {\bf 232} (1977), no.1, 32--35.

\bibitem[Na01]{Natterer} F.\ Natterer. The mathematics of computerized tomography, volume 32 of Classics in Applied Mathematics. SIAM, Philadelphia, PA, 2001. Reprint of the 1986 original.

\bibitem[Pa99]{Pa} G.P. Paternain, {\it Geodesic flows,} Progress in Mathematics {\bf 180}, Birkh\"auser 1999.

\bibitem[PS18]{PS18} G.P. Paternain, M. Salo, {\it Carleman estimates for geodesic X-ray transforms}, arXiv:1805.02163.

%\bibitem[PSU12]{PSU2} G.P. Paternain, M. Salo, G. Uhlmann, {\it The attenuated ray transform for connections and Higgs fields},  Geom. Funct. Anal. {\bf 22} (2012) 1460--1489.


\bibitem[PSU13]{PSU1} G.P. Paternain, M. Salo and G. Uhlmann, {\it Tensor tomography on simple surfaces},  Invent. Math. {\bf 193} (2013) 229--247.

%\bibitem[PSU14a]{PSU_range} G.P. Paternain, M. Salo, G. Uhlmann, {\it On the range of the attenuated ray transform for unitary connections,} Int. Math. Res. Not. (to appear).

%\bibitem[PSU14b]{PSU4} G.P. Paternain, M. Salo, G. Uhlmann, {\it Tensor tomography: progress and challenges,} Chinese Ann. Math. Ser. B {\bf 35} (2014) 399--428.

%\bibitem[PSU14c]{PSU3} G.P. Paternain, M. Salo, G. Uhlmann, {\it Spectral rigidity and invariant distributions on Anosov surfaces,} J. Differential Geometry {\bf 98} (2014) 147--181.

\bibitem[PSU15]{PSU_hd} G.P. Paternain, M. Salo, G. Uhlmann, {\it Invariant distributions, Beurling transforms and tensor tomography in higher dimensions}, Math. Ann. {\bf 363} (2015) 305--362.

\bibitem[PSU20]{PSU_book} G.P. Paternain, M. Salo and G. Uhlmann, {\it Geometric inverse problems in two dimensions}, textbook in progress.



\bibitem[PS88]{PS} L. Pestov, V.A. Sharafutdinov,
{\it Integral geometry of tensor fields on a manifold of negative curvature,}
Siberian Math. J. {\bf 29} (1988) 427--441.

%\bibitem[PU05]{PestovUhlmann} L.\ Pestov, G.\ Uhlmann,
%{\it Two dimensional compact simple Riemannian manifolds are
%boundary distance rigid},
%Ann. of Math. {\bf 161} (2005), 1089--1106.

\bibitem[Pe06]{Petersen} P.\ Petersen, Riemannian geometry. 2nd edition, Springer, 2006.

%\bibitem[SaU11]{SaU11} M. Salo, G. Uhlmann, {\it The attenuated ray transform on simple surfaces,} J. Differential Geom. {\bf 88} (2011) 161--187.



\bibitem[Sh94]{Sh} V.A. Sharafutdinov, {\it Integral geometry of tensor fields,} Inverse and Ill-Posed Problems Series. VSP, Utrecht, 1994.

\bibitem[Sh95]{Sh95} V.A. Sharafutdinov, {\it Finiteness theorem for the ray transform on a Riemannian manifold,} Inverse Problems {\bf 11} (1995) 1039--1050.

\bibitem[Sh07]{Sh07} V. A. Sharafutdinov, {\it Variations of Dirichlet-to-Neumann map and deformation boundary rigidity of simple 2-manifolds,} J. Geom. Anal. {\bf 17} (2007) 147--187. 

%\bibitem[Sh00]{Sha_connection} V.A. Sharafutdinov, {\it On an inverse problem of determining a connection on a vector bundle,}
%J. Inverse Ill-Posed Probl. {\bf 8} (2000) 51--88.

\bibitem[SU04]{SU04} P. Stefanov, G. Uhlmann, {\it Stability estimates for the X-ray transform of tensor fields and boundary rigidity}, Duke Math J. {\bf 123} (2004) 445--467.


%\bibitem[SU05]{SU3} P. Stefanov and G. Uhlmann, {\it Boundary rigidity and stability for generic simple metrics}, J. Amer. Math. Soc., \textbf{18} (2005) 975--1003.

%\bibitem[SU09]{SU09} P. Stefanov, G. Uhlmann, {\it Local lens rigidity with incomplete data for a class of non-simple Riemannian manifolds,} J. Differential Geom. {\bf 82} (2009) 383--409.

\bibitem[SUV18]{SUV_tensor}
P. Stefanov, G. Uhlmann, A. Vasy, {\it Inverting the local geodesic X-ray transform on tensors,}  J. Anal. Math. {\bf 136} (2018) 151--208.

%\bibitem[Ta12]{Tao_randommatrix} T.\ Tao, {\it Topics in random matrix theory,} Graduate Studies in Mathematics {\bf 132}. AMS, 2012.

\bibitem[UV16]{UhlmannVasy} G.\ Uhlmann, A.\ Vasy, \textit{The inverse problem for the local geodesic ray transform},  Invent. Math. {\bf 205} (2016) 83--120.

%\bibitem[Ve92]{V} L.B. Vertgeim, {\it Integral geometry with a matrix weight, and a nonlinear problem
%of recovering matrices,} Sov. Math. Dokl. {\bf 44} (1992) 132--135.


\end{thebibliography}

\end{document}